\documentclass[12pt, reqno]{amsart}
\usepackage{amsmath, amsthm, amsxtra, amscd, amsfonts, amssymb, mathrsfs, graphicx, color,ulem, pstricks}

\usepackage{pgfplots}
\pgfplotsset{compat=1.15}
\usepackage{mathrsfs}
\usetikzlibrary{arrows}

\usepackage[bookmarksnumbered, colorlinks, plainpages]{hyperref}
\hypersetup{colorlinks=true,linkcolor=red, anchorcolor=green, citecolor=cyan, urlcolor=red, pdftoolbar=true}
\usepackage{setspace}
\theoremstyle{plain}
\newtheorem{theorem}{Theorem}[section]
\newtheorem{lemma}[theorem]{Lemma}

\newtheorem{corollary}[theorem]{Corollary}
\theoremstyle{definition}

\newtheorem{example}[theorem]{Example}

\theoremstyle{remark}
\newtheorem{remark}[theorem]{Remark}
\numberwithin{equation}{section}

\definecolor{darkgreen}{rgb}{.1,.5,0}

\theoremstyle{plain}

\theoremstyle{definition}

\newtheorem*{Sketch of proof}{Sketch of proof}

\numberwithin{equation}{section}
\begin{document}
	\setcounter{page}{1}
	\title[Cauchy dual  Subnormality of extensions of $m$-isometric  composition operators  ]{Cauchy dual  Subnormality of extensions of $m$-isometric  composition operators on directed graphs}
	
	\author[ V. Devadas, T.  Prasad,  E. Shine Lal ]{ V. Devadas, T.  Prasad and E. Shine Lal }
	
	\address{V. Devadas\endgraf
		Department of Mathematics, Sree Narayana College Alathur, Affiliated to University of Calicut,
		Kerala-678682, 
		India.}
	\email{\textcolor[rgb]{0.00,0.00,0.84}{ v.devadas.v@gmail.com}}
	
	\address{T. Prasad\endgraf
	Department of Mathematics,  Nalanda University, Rajgir-803116, District- Nalanda , Bihar, India}.
	\email{\textcolor[rgb]{0.00,0.00,0.84}{ prasadvalapil@gmail.com}}
	
	\address{E. Shine Lal\endgraf
		Department of Mathematics, University College Thiruvananthapuram,
		Kerala-695034, 
		India.}
	\email{\textcolor[rgb]{0.00,0.00,0.84}{ shinelal.e@gmail.com}}


	\begin{abstract}
		In this paper, we discuss the analyticity,  $\Delta_{C_\phi}$- regularity, kernel condition and subnormality of Cauchy dual of $k$-quasi-$m$-isometric  composition operators on directed graphs with one circuit and multiple branching vertices. 
	\end{abstract}
	\maketitle
	\section{Introduction and preliminaries }
	The class of $m$-isometries, introduced by Agler and Stankus
	\cite{as1,as2,as3}, has been an active topic of research in operator theory due to its rich structural properties and connections with dilation theory.  A natural generalization of this notion is furnished by the class of $k$-quasi-$m$-isometries, which was subsequently investigated in \cite{Mechpra}. The further development for this class have been found in \cite{a,b,c,d,TP1,DSP,DSP2}. Throughout the paper, \(\mathbb N\), \(\mathbb Z_+\), \(\mathbb Z\), \(\mathbb R\), and \(\mathbb C\) denote the sets of positive integers, nonnegative integers, integers, real numbers, and complex numbers, respectively.  Let \(B(\mathcal H)\) denote the algebra of all bounded linear operators on \(\mathcal H\). For \(T\in B(\mathcal H)\) and \(m\in\mathbb N\), let 
	$
	\mathcal{B}_m(T)
	=
	\sum_{j=0}^{m}
	(-1)^j
	\binom{m}{j}
	T^{*(m-j)}T^{m-j}
	$. 
	An operator \(T\in B(\mathcal H)\) is called an \(m\)-isometry if
	$
	\mathcal{B}_m(T)=0.
	$
	More generally,  for \(k,m\in\mathbb N\), \(T\) is said to be a \(k\)-quasi-\(m\)-isometry whenever
	$
	T^{*k}\mathcal{B}_m(T)T^k=0 
	$
	(see \cite{Mechpra}).  If $k=1$, then  $T$ is of quasi-$m$-isometry.
		 Several notions associated with left-invertible operators play an important role in the study of \(k\)-quasi-\(m\)-isometries and their Cauchy duals. Recall that an operator \(T\in B(\mathcal H)\) is said to be analytic if
		 $
		 \mathcal{R}^\infty(T)=\bigcap_{n=0}^{\infty}T^{n}(\mathcal H)=\{0\}.
		 $
		 Analytic operators were introduced and studied extensively in the context of operator models and invariant subspace theory (see \cite{as1,as2,as3,ZJJK }).
		 
		 If $T\in B(\mathcal H)$ is left invertible, then there exists $S\in B(\mathcal H)$ such that $ST=I$. In this case, $T$ is injective, has closed range, and $T^{*}T$ is invertible. The associated Cauchy dual operator, introduced by Shimorin \cite{Shm}, is defined by, 
		 $
		 T'=T(T^{*}T)^{-1}.
		 $
		 It follows that
		 $
		 (T')^{*}=(T^{*}T)^{-1}T^{*}$ and 
		$ (T')^{*}T=I
		 $
		 (see \cite{g2,Shm}).   Following \cite{g2,cbls},  an operator $T\in B(\mathcal H)$ is said to be $\Delta_T$-regular if
		 $
		 \Delta_TT=\Delta_T^{1/2}T\Delta_T^{1/2}, 
		 $ where  $
		 \Delta_T=T^{*}T-I.
		 $
		 Another condition that frequently arises in the study of Cauchy dual operators is the kernel condition 
		 $
		 T^{*}T\big(\mathcal N(T^{*})\big)\subseteq \mathcal N(T^{*})$  (see \cite{g2}).
		 
	Consider  $(X,\mathcal F,\mu)$ is a discrete measure space, where $X$ is a countably infinite set and $\mu(\{x\})>0$ for every $x\in X$. A mapping  $\phi:X\to X$ is said to be a measurable transformation if $\phi^{-1}(S) \in \mathcal F$ ~~ for all ~~ $S\in\mathcal F$. 	A measurable transformation $\phi$ is called nonsingular whenever $\mu\circ\phi^{-1}(S)= 0$ if $\mu(S) = 0, S\in\mathcal F. $ Then  $\mu\circ\phi^{-1}$ is a measure defined by
	$
	(\mu\circ\phi^{-1})(S)=\mu(\phi^{-1}(S)),~~ S\in\mathcal F.
	$
	In this case, the Radon--Nikodym derivative of $\mu\circ\phi^{-1}$ with respect to $\mu$ is denoted by $h$. More generally, for $n\in\mathbb Z_{+}$, the Radon--Nikodym derivative of $\mu\circ\phi^{-n}$ with respect to $\mu$ is denoted by $h_n$, where $h_0=1$ and $h_1=h$.  For a nonsingular transformation $\phi$, the composition operator $C_\phi$ on $L^2(\mu)$ is defined by
	$C_\phi f=f\circ\phi$ for $ f \in L^2(\mu).$
	It is well known that $C_\phi$ is bounded if and only if $h\in L^\infty(\mu)$. In this case,
	$\displaystyle \|C_\phi\|^2=\|h\|_\infty, \quad \|C_\phi^{\,n}f\|^2
	= \int_X h_n |f|^2\,d\mu,\quad n\in\mathbb Z_+.
	$  A multiplication operator  $M_\pi$ induced by $\pi\in L^\infty(\mu)$  is defined by  $M_\pi f=\pi f$ for $f \in L^2(\mu)$. The weighted composition operator induced by $\phi$ and $\pi$ is
	$Wf=\pi(f\circ\phi),~~ f\in L^2(\mu).$ If $\pi_n=\prod_{j=0}^{n-1}\pi\circ\phi^j,$ then $W^n f=\pi_n(f\circ\phi^n),~~ n\in\mathbb Z_+.$

	Now we give the  basic structure of directed graph setting that was considered in \cite{DSP2}.  Let $\kappa\in \mathbb{N}$, $\eta_r\in\mathbb{Z_+\cup\{\infty\}}$  for   $r\in J_{[1,\kappa]}$ and at least one of $\eta_r$ is nonzero. Throughout this  paper  we let 
	$ \displaystyle	X= X_\kappa \cup\bigcup_{r=1}^{\kappa} X_{\eta_r},$
	as a directed graph containing a single circuit,
	where $X_\kappa=\{x_1,x_2,\ldots,x_k\} $ is the set of branching vertices on this circuit, 
	$	X_{\eta_r}=\bigcup_{i=1}^{\eta_r} \{x^r_{i,j}: j\in\mathbb{N}\}$ is the set of all vertices along the $i^{\text{th}}$ branch emanating from $x_r$ for $i\in J_{[1,\eta_r]}$ and $\eta_r$ denotes the number of branches originating from the vertex $x_r$.  Figure 1 illustrates this construction for the specific case $\kappa = 4$ and $\eta_{r}=2$ for all $r\in J_{[1,\kappa]}$.
	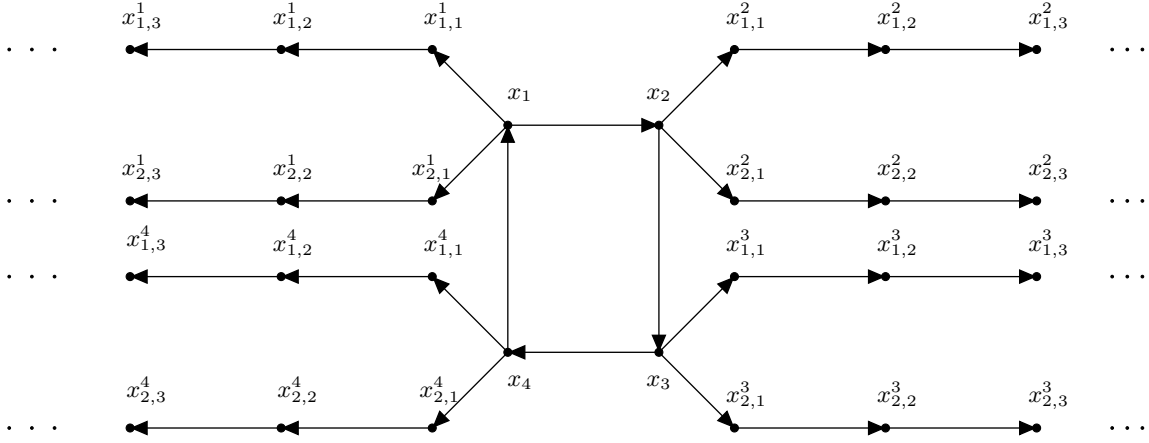
\begin{figure}
		\centering
		\begin{tikzpicture}[line cap=round,line join=round,>=triangle 45,x=1cm,y=1cm]
			\clip(-8,-3.28) rectangle (13.72,5.24);
			\draw [->,line width=.5pt] (-1,2) -- (1,2);
			\draw [->,line width=.5pt] (1,2) -- (1,-1);
			\draw [->,line width=.5pt] (1,-1) -- (-1,-1);
			\draw [->,line width=.5pt] (-1,-1) -- (-1,2);
			\draw [->,line width=.5pt] (1,2) -- (2,3);
			\draw [->,line width=.5pt] (1,2) -- (2,1);
			\draw [->,line width=.5pt] (2,1) -- (4,1);
			\draw [->,line width=.5pt] (4,1) -- (6,1);
			\draw [->,line width=.5pt] (1,-1) -- (2,0);
			\draw [->,line width=.5pt] (2,0) -- (4,0);
			\draw [->,line width=.5pt] (4,0) -- (6,0);
			\draw [->,line width=.5pt] (-1,2) -- (-2,3);
			\draw [->,line width=.5pt] (-2,3) -- (-4,3);
			\draw [->,line width=.5pt] (-4,3) -- (-6,3);
			\draw [->,line width=.5pt] (-1,2) -- (-2,1);
			\draw [->,line width=.5pt] (-2,1) -- (-4,1);
			\draw [->,line width=.5pt] (-4,1) -- (-6,1);
			\draw [->,line width=.5pt] (2,3) -- (4,3);
			\draw [->,line width=.5pt] (4,3) -- (6,3);
			\draw [->,line width=.5pt] (1,-1) -- (2,-2);
			\draw [->,line width=.5pt] (2,-2) -- (4,-2);
			\draw [->,line width=.5pt] (4,-2) -- (6,-2);
			\draw [->,line width=.5pt] (-1,-1) -- (-2,0);
			\draw [->,line width=.5pt] (-2,0) -- (-4,0);
			\draw [->,line width=.5pt] (-4,0) -- (-6,0);
			\draw [->,line width=.5pt] (-1,-1) -- (-2,-2);
			\draw [->,line width=.5pt] (-2,-2) -- (-4,-2);
			\draw [->,line width=.5pt] (-4,-2) -- (-6,-2);
			\begin{scriptsize}
				\draw [fill=black] (-1,2) circle (1.5pt);
				\draw[color=black] (-0.84,2.4) node {$x_1$};
				\draw [fill=black] (1,2) circle (1.5pt);
				\draw[color=black] (1,2.4) node {$x_2$};
				\draw [fill=black] (1,-1) circle (1.5pt);
				\draw[color=black] (1,-1.4) node {$x_3$};
				\draw [fill=black] (-1,-1) circle (1.5pt);
				\draw[color=black] (-0.84,-1.4) node {$x_4$};
				\draw [fill=black] (2,3) circle (1.5pt);
				\draw[color=black] (2.16,3.43) node {$x_{1,1}^2$};
				\draw [fill=black] (2,1) circle (1.5pt);
				\draw[color=black] (2.16,1.43) node {$x_{2,1}^2$};
				\draw [fill=black] (4,3) circle (1.5pt);
				\draw[color=black] (4.16,3.43) node {$x_{1,2}^2$};
				\draw [fill=black] (6,3) circle (1.5pt);
				\draw[color=black] (6.16,3.43) node {$x_{1,3}^2$};
				\draw [fill=black] (7,3) circle (.5pt);
				\draw [fill=black] (7.2,3) circle (.5pt);
				\draw [fill=black] (7.4,3) circle (.5pt);
				\draw [fill=black] (7,1) circle (.5pt);
				\draw [fill=black] (7.2,1) circle (.5pt);
				\draw [fill=black] (7.4,1) circle (.5pt);
				\draw [fill=black] (7,0) circle (.5pt);
				\draw [fill=black] (7.2,0) circle (.5pt);
				\draw [fill=black] (7.4,0) circle (.5pt);
				\draw [fill=black] (7,-2) circle (.5pt);
				\draw [fill=black] (7.2,-2) circle (.5pt);
				\draw [fill=black] (7.4,-2) circle (.5pt);
				\draw [fill=black] (-7,3) circle (.5pt);
				\draw [fill=black] (-7.3,3) circle (.5pt);
				\draw [fill=black] (-7.6,3) circle (.5pt);
				\draw [fill=black] (-7,1) circle (.5pt);
				\draw [fill=black] (-7.3,1) circle (.5pt);
				\draw [fill=black] (-7.6,1) circle (.5pt);
				\draw [fill=black] (-7,0) circle (.5pt);
				\draw [fill=black] (-7.3,0) circle (.5pt);
				\draw [fill=black] (-7.6,0) circle (.5pt);
				\draw [fill=black] (-7,-2) circle (.5pt);
				\draw [fill=black] (-7.3,-2) circle (.5pt);
				\draw [fill=black] (-7.6,-2) circle (.5pt);
				
				\draw [fill=black] (4,1) circle (1.5pt);
				\draw[color=black] (4.16,1.43) node {$x_{2,2}^2$};
				\draw [fill=black] (6,1) circle (1.5pt);
				\draw[color=black] (6.16,1.43) node {$x_{2,3}^2$};
				\draw [fill=black] (2,0) circle (1.5pt);
				\draw[color=black] (2.16,0.43) node {$x_{1,1}^3$};
				\draw [fill=black] (4,0) circle (1.5pt);
				\draw[color=black] (4.16,0.43) node {$x_{1,2}^3$};
				\draw [fill=black] (6,0) circle (1.5pt);
				\draw[color=black] (6.16,0.43) node {$x_{1,3}^3$};
				\draw [fill=black] (-2,3) circle (1.5pt);
				\draw[color=black] (-1.84,3.43) node {$x_{1,1}^1$};
				\draw [fill=black] (-4,3) circle (1.5pt);
				\draw[color=black] (-3.84,3.43) node {$x_{1,2}^1$};
				\draw [fill=black] (-6,3) circle (1.5pt);
				\draw[color=black] (-5.84,3.43) node {$x_{1,3}^1$};
				\draw [fill=black] (-2,1) circle (1.5pt);
				\draw[color=black] (-2,1.43) node {$x_{2,1}^1$};
				\draw [fill=black] (-4,1) circle (1.5pt);
				\draw[color=black] (-3.84,1.43) node {$x_{2,2}^1$};
				\draw [fill=black] (-6,1) circle (1.5pt);
				\draw[color=black] (-5.84,1.43) node {$x_{2,3}^1$};
				
				\draw [fill=black] (2,-2) circle (1.5pt);
				\draw[color=black] (2.16,-1.57) node {$x_{2,1}^3$};
				\draw [fill=black] (4,-2) circle (1.5pt);
				\draw[color=black] (4.16,-1.57) node {$x_{2,2}^3$};
				\draw [fill=black] (6,-2) circle (1.5pt);
				\draw[color=black] (6.16,-1.57) node {$x_{2,3}^3$};
				\draw [fill=black] (-2,0) circle (1.5pt);
				\draw[color=black] (-1.84,0.43) node {$x_{1,1}^4$};
				\draw [fill=black] (-4,0) circle (1.5pt);
				\draw[color=black] (-3.84,0.43) node {$x_{1,2}^4$};
				\draw [fill=black] (-6,0) circle (1.5pt);
				\draw[color=black] (-5.78,0.49) node {$x_{1,3}^4$};
				\draw [fill=black] (-2,-2) circle (1.5pt);
				\draw[color=black] (-1.9,-1.51) node {$x_{2,1}^4$};
				\draw [fill=black] (-4,-2) circle (1.5pt);
				\draw[color=black] (-3.78,-1.51) node {$x_{2,2}^4$};
				\draw [fill=black] (-6,-2) circle (1.5pt);
				\draw[color=black] (-5.78,-1.51) node {$x_{2,3}^4$};
			\end{scriptsize}
		\end{tikzpicture}
		\caption{Directed graph with one circuit and more than one branching vertex}
	\end{figure}\label{fig:Directed graph with one circuit}
	
	Let $\mathcal{F}$ be the set of all subsets of $X$ and $\mu(x)> 0$ ~for all~~ $ x \in X$. By using the directed graph described above, we  have  the corresponding parent function as follows:

	\begin{align*}
		par(x)= \left\{
		\begin{array}{ll}\
			x^r_{i,j}, & \mathrm{if}~~ x=x^r_{i,j+1} ~~~\mathrm{for}~~r\in J_{[1,\kappa]},~~~ i\in J_{[1,\eta_r]},~~ \mathrm{and}~~j\in \mathbb{N} ,\\\\
			x_r, & \mathrm{if}~~ x=x^s_{i,j}, ~~~\mathrm{for}~~ s\in J_{[1,\kappa]}  \mathrm{~and}~ \Phi_2(1+r)=\Phi_2(s+j),~  j\in\mathbb{N}, \\
			& i\in J_{[1,\eta_s]},~~ \mathrm{or}~~x=x_{\Phi_2(1+r)}.
		\end{array}\right. ,\end{align*} 
	where $\Phi_2:\mathbb{Z}\to J_{[1,\kappa]}$ is a function satisfies the following  condition with  $\kappa$ and a unique function $\Phi_1: \mathbb{Z}\to \mathbb{Z} $ as
	$
	p=\Phi_1(p)\kappa+\Phi_2(p),\quad p\in\mathbb{Z}.
	$
	The properties of $\Phi_1$  and $\Phi_2 $ are given below
	$
	\Phi_1(l\kappa+1)=\Phi_1(l\kappa+r),\quad l\in\mathbb{Z},~ r\in J_{[1,\kappa]},
	$ and 
	$
	\Phi_2(l\kappa+r_1+r_2)=\Phi_2(l\kappa+r_1)+r_2,~~ l\in\mathbb{Z}, ~~ r_1\in\mathbb{N},~ r_2\in\mathbb{Z}_+,~ r_1+r_2\in J_{[1,\kappa]}.  
	$
	

	\begin{equation}\label{equ1} 
		\left.
		\begin{aligned}
			\text{Assume that }   \varphi : X \to X \text{ is a  measurable transformation defined by} \\
			\varphi(x) = \operatorname{par}(x), \; x \in X \quad \quad \quad \quad \quad \quad \quad \quad \quad \quad \quad \quad\quad
		\end{aligned}
		\right\}
	\end{equation}
	
	Since the transformation $\phi$ is nonsingular,  $\phi^p$ is also nonsingular for $p \in \mathbb{N}$. So, the Radon-Nikodym derivative $h_p = \frac{d(\mu \circ \phi^{-p})}{d\mu}$ can be written as follows (see \cite{DSP}):
	$$h_p(x)= \left\{
	\begin{array}{ll}
		\frac{\mu(x^r_{i,j+p})}{\mu(x^r_{i, j})}, & \mathrm{if}~~ x=x^r_{i,j},  ~~ r\in J_{[1,\kappa]},  ~~  i\in J_{[1,\eta_r]}, \\
		& ~~j\in \mathbb{N},\\\\
		\frac{\mu(x_{\Phi_{2}(p+r)})
			+ \displaystyle\sum_{j=1}^{p}
			\sum_{\substack{s=1\\\Phi_{2}(p+r)=\Phi_{2}(s+j)}}^{\kappa}
			\sum_{i=1}^{\eta_{s}}
			\mu(x^{s}_{i,j})}{\mu(\{x_r\})},	 &\mathrm{if}~~ x=x_r,~~r\in J_{[1,\kappa]} .\\\\
	\end{array}\right.$$

	 In this note, we investigate the analyticity and subnormality of the Cauchy dual associated with the class of $k$-quasi-$m$-isometric composition operators acting on directed graphs containing a single circuit and multiple branching vertices. This study is motivated by the work of Jabłoński and Kośmider \cite{ZJJK } on $m$-isometric composition operators on directed graphs, as well as recent developments concerning the subnormality of Cauchy duals of $2$-isometries\cite{MB}. We have investigated these concepts and extended them in two directions: first, to a broader graph-theoretic setting, and second, to higher classes of $m$-isometries. These extensions are established through Theorems \ref{Thm1}, \ref{Tm2}, \ref{Tm3}, \ref{Tm4}, and \ref{Tm5}. Furthermore, in the case $\kappa=1$, Example \ref{eg1} demonstrates that the analyticity of the composition operator $C_\phi$ does not necessarily entail its $1$-quasi-$2$-isometry, while Example \ref{eg2} reveals that a  quasi-$2$-isometric composition operator $C_\phi$ need not satisfy the $\Delta_{C_\phi}$-regularity condition. Moreover, for $\kappa=2$, Example \ref{eg3} establishes that the kernel condition associated with $C_\phi$ is insufficient to ensure $\Delta_{C_\phi}$-regularity. Collectively, these examples highlight the subtle distinctions among analyticity, quasi-$2$-isometry, kernel conditions, and regularity properties.

\section{Subnormality of Cauchy dual of composition operators on directed graphs }
	We begin this secion by studying  the analyticity of  $k$-quasi-$m$-isometric  composition operators on directed graphs. First we extend the \cite[Proposition 4.2]{ZJJK } by using the following  the extended version of  \cite[Lemma 4.1]{ZJJK }.
	
 \begin{lemma}\label{lm1}
Assume that \eqref{equ1} holds, $C_\phi$ is bounded operator on $L^2(\mu)$ and  $f\in L^2(\mu)$. Then the following statements are equivalent:
  \begin{enumerate}
  	\item $f \in \mathcal{R}^\infty(C_\phi),$
  	\item $f(x_r)= f(x^s_{i, j})$~~ for all ~~~ $r, s \in J_{[1, \kappa]},~~ j\in\mathbb{N},~~ r= \Phi_{2}(s+j), ~~ i \in J_{[1, \eta_{s}]},$
  	\item $ C^{\kappa}_{\phi} f = f .$
  \end{enumerate}
 \end{lemma}
\begin{proof}
	Assume that \eqref{equ1} is true and $C_\phi \in B(L^2(\mu))$ and let $f\in L^2(\mu)$.
	Then for $p\geq 1 $, we have 
	\begin{align*}
		\phi^p(x)= \left\{
		\begin{array}{ll}
			x^r_{i,j}, & \mathrm{if}~~ x=x^r_{i,j+p} ~~~\mathrm{for}~~r\in J_{[1,\kappa]},~~~ i\in J_{[1,\eta_r]},~~j\in \mathbb{N} ,\\\\
			x_r, & \mathrm{if}~~x=x_{\Phi_2(p+r)} ~~ \mathrm{or}~~  x=x^s_{i,j},~~ r, s\in J_{[1,\kappa]},  \\
			& \Phi_2(p+r)=\Phi_2(s+j),~  j\in J_{[1,p]}, ~~i\in J_{[1,\eta_s]}.
		\end{array}\right. \end{align*} 
	(i)$\implies$ (ii). Assume that $f \in \mathcal{R}^\infty(C_\phi)$. Then for each $p\geq 1$ there exists a function $g_p \in L^2(\mu)$ such that $f=g_p\circ \phi^p$.  Thus for $p\geq 1$ we get,
	$$f(x_{\Phi_{2}(p+r)})= f(x^s_{i, j}) ~~\textrm{for all} ~~ r, s \in J_{[1, \kappa]},~~ ~~ \Phi_{2}(p+r)= \Phi_{2}(s+j),~~ j\in J_{[1,p]},~~ i \in J_{[1, \eta_{s}]}.$$
	Hence we get , 
	$$f(x_r)= f(x^s_{i, j})~~~ \textrm{for all}~~ r, s \in J_{[1, \kappa]},~~ ~~ r = \Phi_{2}(s+j),~~ j\in \mathbb{N},~~ i \in J_{[1, \eta_{s}]}.$$

(ii)$\implies$ (iii). For $x\in X$, consider $(C^{\kappa}_{\phi} f)(x) = (f\circ \phi^\kappa)(x) $
	
	\begin{align*}
		(f\circ \phi^\kappa)(x)= \left\{
		\begin{array}{ll}
			f(x^r_{i,j}), & \mathrm{if}~~ x=x^r_{i,j+\kappa} ~~r\in J_{[1,\kappa]},~~~ i\in J_{[1,\eta_r]},~~j\in \mathbb{N} ,\\\\
			f(x_r), & \mathrm{if}~~x=x_{\Phi_2(\kappa+r)} ~~ \mathrm{or}~~  x=x^s_{i,j},~~ r, s\in J_{[1,\kappa]}  \mathrm,   \\
			& \Phi_2(\kappa+r)=\Phi_2(s+j),~  j\in J_{[1,\kappa]}, ~~i\in J_{[1,\eta_s]}.
		\end{array}\right. \end{align*} 
	Since $\Phi_2(\kappa+r)=r$ and $\Phi_2(l\kappa+r)=r ~~ \text{for}~~ l \in \mathbb{Z_+}, r \in J_{[1, \kappa]}$,  by (ii)  it follows that  $C^{\kappa}_{\phi} f =f. $
	
	(iii)$\implies$ (i). Let $p\in \mathbb{N}$. Then $p= l\kappa+r$ for $  ~~l \in \mathbb{Z_+}$ and $ ~~ r \in J_{[1, \kappa-1]}.$ 

	Now, $C^{p}_{\phi} f = C^{l\kappa+r}_{\phi} f = C^{l\kappa}_{\phi}C^{r}_{\phi}(f)= (C^{r}_{\phi})(f)$. Therefore, if we choose $g= f\circ \phi^{\kappa-r}$, then  we get 
 $$C^{p}_{\phi}(g)= C^{r}_{\phi}(f\circ \phi^{\kappa-r})= f\circ \phi^{\kappa}= C^{\kappa}_{\phi}f = f.$$ \\
 Which  implies $f \in \mathcal{R}^p(C_\phi), ~~ p \in \mathbb{N}.$
 Hence $f \in \mathcal{R}^\infty(C_\phi)$.
\end{proof}
\begin{theorem}\label{Thm1}
	Assume that \eqref{equ1} holds and  $C_\phi$ is a bounded operator on $L^2(\mu)$. 
	Then $C_\phi$ is analytic if and only if the series  
  $$\displaystyle\sum_{j=1}^{\infty}
 \sum_{\substack{s=1\\ r=\Phi_{2}(s+j)}}^{\kappa}
 \sum_{i=1}^{\eta_{s}}
 \mu(x^{s}_{i,j})$$ 
 is divergent for every $ r \in J_{[1, \kappa]}.$
		
\end{theorem}
	
\begin{proof}
	Suppose  that \eqref{equ1} holds and $C_\phi$ is a analytic operator on $L^2(\mu)$. Then $\mathcal{R}^{\infty}(C_\phi)$ is a trivial space.  Let 
	$\displaystyle\sum_{j=1}^{\infty}
	\sum_{\substack{s=1\\ r=\Phi_{2}(s+j)}}^{\kappa}
	\sum_{i=1}^{\eta_{s}}
	\mu(x^{s}_{i,j})$ 
	is convergent for some $ t \in J_{[1, \kappa]}.$
	
	 Consider  the  function $ f_t: X \rightarrow \mathbb{C}$  given by 
		\begin{align*}
		f_t(x)= \left\{
		\begin{array}{ll}
			1, & \mathrm{if},x=x_t ~~ \mathrm{or}~~  x=x^s_{i,j}, ~~  s\in J_{[1,\kappa]}     \\
			& t=\Phi_2(s+j),~  j\in \mathbb{N}, ~~i\in J_{[1,\eta_s]}.\\\\
			0, & otherwise.
		\end{array}\right. \end{align*} 
	Then by Lemma \ref{lm1} $f_t \in \mathcal{R}^{\infty}(C_\phi)$ and it is a nonzero function  whose support is 
	$$
		supp(f_t)=\overline{\{x\in X : f_t(x)\neq 0\}}= \{x_{t}\} \cup\displaystyle\bigcup_{j=1}^{\infty}
		\bigcup_{\substack{s=1\\ t =\Phi_{2}(s+j)}}^{\kappa}
		\bigcup_{i=1}^{\eta_{s}}
		\{x^{s}_{i,j}\}.
	$$
	Thus, 
	$$\int_{X} |f_t|^2 d\mu = \mu(x_t)+\displaystyle\sum_{j=1}^{\infty}
	\sum_{\substack{s=1\\ t=\Phi_{2}(s+j)}}^{\kappa}
	\sum_{i=1}^{\eta_{s}}
	\mu(x^{s}_{i,j}) < \infty. $$
	Therefore, $f_t \in L^2(\mu) \cap \mathcal{R}^{\infty}(C_\phi)$ , a contradiction to analyticity of $C_\phi .$
	
	Conversely Suppose That $\displaystyle\sum_{j=1}^{\infty}
	\sum_{\substack{s=1\\ r=\Phi_{2}(s+j)}}^{\kappa}
	\sum_{i=1}^{\eta_{s}}
	\mu(x^{s}_{i,j})$ 
	is divergent for every $ r \in J_{[1, \kappa]}.$
	Define a function $f: X \rightarrow \mathbb{C}$ by $ f= \displaystyle \sum_{r=1}^{\kappa}f(x_r)f_r$,  where  $f_r$ is the  characteristic functions of the set  $\{x_{r}\} \cup\displaystyle\bigcup_{j=1}^{\infty}
	\bigcup_{\substack{s=1\\ t =\Phi_{2}(s+j)}}^{\kappa}
	\bigcup_{i=1}^{\eta_{s}}
	\{x^{s}_{i,j}\}$ for $r \in J_{[1, \kappa]}$ and $supp(f_r) \cup supp(f_t)= \emptyset, ~~\textrm{for}~~ r\neq t \in J_{[1, \kappa]}.$ Then by Lemma \ref{lm1} $f \in \mathcal{R}^{\infty}(C_\phi) \cap L^2(\mu)$. 
	But note that 
	$$\int_{X} |f|^2 d\mu = \displaystyle \sum_{r=1}^{\kappa} f(x_r)\left[\int_{X} |f_r|^2 d\mu\right] = \displaystyle \sum_{r=1}^{\kappa} f(x_r)\left[\mu(x_r)+\sum_{j=1}^{\infty}
	\sum_{\substack{s=1\\ r=\Phi_{2}(s+j)}}^{\kappa}
	\sum_{i=1}^{\eta_{s}}
	\mu(x^{s}_{i,j})\right] < \infty. $$
	Then by assumption $f(x_r)=0~~~ \textrm{for all}~~ r \in J_{[1, \kappa]}$. Consequently,  $f(x^{s}_{i,j})= 0 ~~ \textrm{for all}~~ s \in J_{[1, \kappa]}, ~~j \in \mathbb{N},$ and $~~i \in J_{[1, \eta_s]}.$ Therefore, $ \mathcal{R}^{\infty}(C_\phi)= \{0\}$. Hence $C_\phi$ is analytic.
	
\end{proof}

\begin{corollary}\label{Cr1}
	Let $k \in \mathbb{Z_+}$ and  $m\geq 2$. Assume that \eqref{equ1} holds and $C_\phi \in B(L^2(\mu)) $ is a $k$-quasi-$m$-isometry. Then $C_\phi $ is analytic.
	
\end{corollary}
\begin{proof}
	Assume that $C_\phi \in B(L^2(\mu)) $ is a $k$-quasi-$m$-isometry satisfying \eqref{equ1}, for $k \in \mathbb{Z_+}$ and  $m\geq 2$. Then by \cite[Theorem 2.6]{DSP}, 
	$\{\mu(x^r_{i,k+j+1})\}_{j=0}^{\infty}$ 
is a polynomial in $j$ of degree at most $m-2$ for all $r\in J_{[1,\kappa]}, ~~i\in J_{[1,\eta_{r}]}$, and
	$$
	\sum_{p=0}^{m} (-1)^p \binom{m}{p}
	\, h_{p+k}(x_r)=0~~\textrm{for all} ~~ r\in  J_{[1,\kappa]}.
	$$
Then $\mu(x^r_{{i,k+j+1}}) $  doesnot converge to zero as $j\rightarrow \infty, ~~r \in J_{[1,\kappa]}, ~~i\in J_{[1,\eta_{r}]}. $	 This gives the divergence of the series
 $\sum_{j=1}^{\infty}
\sum_{\substack{s=1\\ r=\Phi_{2}(s+j)}}^{\kappa}
\sum_{i=1}^{\eta_{s}}
\mu(x^{s}_{i,j})$  for every $r \in J_{[1,\kappa]}$. Then by Theorem \ref{Thm1} we get $C_\phi$ is analytic.
\end{proof}
The following example shows that converse of Corollory \ref{Cr1} is not true. 
\begin{example}\label{eg1}
	Assume that \eqref{equ1} holds and $C_\phi \in \mathcal{B}(L^2(\mu))$. Define 
	\begin{align*}
		\mu(x)= \left\{
		\begin{array}{ll}
			1, & \mathrm{if},~~~ x=x_r, \quad   r \in J_{[1,\kappa]}   \\
			\displaystyle\frac{s^i}{j},& if, ~~ x= x^s_{i, j}, \quad s \in J_{[1,\kappa]}, ~~~i \in J_{[1,\eta_{s}]}, ~~~j \in \mathbb{N}. \\\\
		\end{array}\right. \end{align*} 
	Then the series 
	 $\displaystyle\sum_{j=1}^{\infty}
	\sum_{\substack{s=1\\ r=\Phi_{2}(s+j)}}^{\kappa}
	\sum_{i=1}^{\eta_{s}}
	\mu(x^{s}_{i,j})$
	is divergent for every $ r \in J_{[1, \kappa]}.$ Therefore, by applying Theorem \ref{Thm1} we have, $C_\phi$ is analytic. On the other hand, $C_\phi$ is not a $k$-quasi-$m$-isometry for any $k \in \mathbb{Z_+}$ and $m\geq2$. Indeed, since $\mu(x^{s}_{i,j}) \longrightarrow 0$ as $j \longrightarrow \infty $ for $s \in J_{[1,\kappa]}, ~~~i \in J_{[1,\eta_{s}]}$, the sequence $\{\mu(x^r_{{i,k+j+1}})\}_{j=0}^{\infty}$ cannot be a polynomial in $j$ of degree at most $m-2$ for any $r\in J_{[1,\kappa]}$ and $i\in J_{[1,\eta_{r}]}.$ Hence, by \cite[Theorem 2.6]{DSP}, $C_\phi$ fails to be a $k$-quasi-$m$-isometry.
	
\end{example}

The following lemma shows that the Cauchy dual of a composition operator $C_\phi \in B(L^2(\mu))$ satisfying \eqref{equ1} itself a weighted composition operator. 

\begin{lemma}\label{lm2}
Assume that \eqref{equ1} holds and $C_\phi \in B(L^2(\mu))$ is left invertible, then the Cauchy dual $C^{'}_\phi$ of $C_\phi$ is a weighted composition operator defined by  
	$$C^{'}_\phi = w_\phi C_\phi, $$ where $w_\phi = \frac{1}{h\circ \phi} \in L^{\infty}(\mu)$.
\end{lemma}
	\begin{proof}
		The reqired result follows by a similar argument as in 	(See \cite{ZJJK ,RKS}) and by using \eqref{equ1}.
	\end{proof}
	\begin{remark}
	If $w_\phi = \frac{1}{h\circ \phi}$ is essentially bounded with respect to $\mu$ as in \eqref{equ1}, then $\mu_{w_\phi}(x)= |w_{\phi}(x)|^2 \mu (x), \quad x\in X $ is positive measure on $X$ and 
	\begin{align}\label{eqn2}
	\widehat{w}_n= \left\{
	\begin{array}{ll}
		1 & n=0 \\
		\prod_{j=0}^{n-1} w \circ \phi^{j} & n\in \mathbb{Z_+}
	\end{array}\right.
	\end{align} 
	
is a sequence of positive numbers. 
	\end{remark}
	\begin{lemma}\label{lm3}
		If \eqref{equ1} holds and $C_\phi \in B(L^2(\mu))$ is quasi-$2$-isometry, then $C_\phi $ is left invertible and its Cauchy dual $C^{'}_\phi$ is a weighted composition operator. Furthermore,  $C^{'}_\phi$ is subnormal if and only if 
		$\left\{\lVert {C^{'}}^n_\phi f\rVert^2 \right\}_{n=0}^{\infty}$ is a Stieltjes moment sequence.
	\end{lemma}
\begin{proof}
	The left invertible property of $C_\phi $ is immediate from the  definition of  $\mu$. The remaining part of the proof  follows by Lemma \ref{lm2} together with \cite[Theorem 49]{BJJS}.
\end{proof}
\begin{lemma}\label{lm4}
	Suppose that \eqref{equ1} holds with $\kappa=1$, $C_\phi \in B(L^2(\mu))$ is quasi-$2$-isometry and  $w_\phi = \frac{1}{h\circ \phi} \in L^{\infty}(\mu)$. Then 
	
	\begin{enumerate}
		\item 
			\begin{align}\label{eqn3}
			\widehat{w}_n(x)= \left\{
			\begin{array}{ll}
				\alpha^n, & x=x_1 \quad \textrm{or} \quad x= x^1_{i, 1}, \quad i \in J_{[1, \eta_1]},\\
				\alpha^{n+1-j}\alpha_i &  x= x^1_{i, j},\quad i \in J_{[1, \eta_1]},\quad j \in J_{[2, n+1]},\\
				1 &  x= x^1_{i, j+n+1}, \quad i \in J_{[1, \eta_1]},\quad j \in \mathbb{N}.
			\end{array}\right.
			\end{align}
		
	\item 
	\begin{align}\label{eqn4}
			h_{\phi^n,\widehat{w}_n}(x)= \left\{
		\begin{array}{ll}
			\alpha^{2n-1}+\frac{\Sigma_{i=1}^{\eta_1}c_i\alpha_i^2\left[\Sigma_{j=2}^n \alpha^{2(n+1-j)} \right]}{\mu(x_1)}, & x=x_1, \\
			\alpha_i &  x= x^1_{i, 1}, i \in J_{[1, \eta_1]},\\
			1 &  x= x^1_{i, j+1}, i \in J_{[1, \eta_1]}, j \in \mathbb{N}.
		\end{array}\right.,
	\end{align}

where $\alpha=\displaystyle \frac{\mu(x_1)}{\mu(x_1)+\sum_{i=1}^{\eta_1} \mu(x^1_{i, 1})}$ and $\alpha_i= \displaystyle \frac{\mu(x^1_{i, 1})}{\mu(x^1_{i, 2})}, ~~i\in J_{[1, \eta_1]}$
	\end{enumerate}
\end{lemma}
\begin{proof}
	Suppose that  $C_\phi \in B(L^2(\mu))$ is $1$-quasi-$2$-isometry satisfying \eqref{equ1} with $\kappa=1$. Let $w=w_\phi= \displaystyle \frac{1}{h\circ \phi}, $
	where $h$ is the radon -Nikodym derivative of $\mu \circ \phi$ with respect to $\mu$. 
	Then by applying \cite[Theorem 2.6]{DSP} we get, $\mu(x^1_{i,j+1})= c_i,$ a constant for every $i\in J_{[1, \eta_1]}$, $j\in \mathbb{N}$ and $h$ can be written as 
	\begin{align}\label{eqn5}
	h(x)= \left\{
	\begin{array}{ll}
		\alpha, & \text{if}\quad x=x_1, \\
		\alpha_i& \text{if}\quad x= x^1_{i, 1},\quad i \in J_{[1, \eta_1]},\\
		1 &  x= x^1_{i, j+1}, \quad i \in J_{[1, \eta_1]},\quad j \in \mathbb{N},
	\end{array}\right.,
	\end{align}
	
 where $\alpha=\displaystyle \frac{\mu(x_1)+\sum_{i=1}^{\eta_1}\mu(x^1_{i, 1})}{\mu(x_1)}$ and $\alpha_i=\displaystyle \frac{c_i}{\mu(x^1_{i,1})}. $ 

Combining \eqref{eqn2} and \eqref{eqn5}, we arrive at \eqref{eqn3}. Therefore, statement (i) holds.
Next we verify (ii). For this consider $h_{\phi^n,\widehat{w}_n}(x)=  \displaystyle \frac{\mu_{\widehat{w}_n}(\phi ^{-n}(x))}{\mu(x)}	$\\
Now,
\begin{align*}
	\begin{array}{ll}
		h_{\phi^n,\widehat{w}_n}(x_1) & =\displaystyle \frac{\mu_{\widehat{w}_n} \left(\{x_1\} \cup \{x^1_{i,j}:  i \in J_{[1, \eta_1]}, j \in \mathbb{N} \} \right)}{\mu(x_1)}\\
		 & = \displaystyle \frac{\alpha^{2n}\mu(x_1)+ \sum_{i=1}^{\eta_1}\alpha^{2n} \mu(x^1_{i,1}) + \sum_{i=1}^{\eta_1}\sum_{j=2}^{n} \alpha^{2(n+1-j)}\alpha_{i}^{2} \mu(x^1_{i,j})	}{\mu(x_1)}\\
		 & =  	\displaystyle \alpha^{2n-1}+\frac{  \sum_{i=1}^{\eta_1} c_i \alpha_{i}^{2} \sum_{j=2}^{n} \alpha^{2(n+1-j)}	}{\mu(x_1)},
	\end{array}
\end{align*}
 
\begin{align*}
	\begin{array}{ll}
		h_{\phi^n,\widehat{w}_n}(x^1_{i,1}) & = \displaystyle \frac{|\widehat{w}_n(x^1_{i, n+1})|^{2}\mu(x^1_{i, n+1})}{\mu(x^1_{i,1})}\\
		& = \displaystyle \frac{c_i \alpha_i^{2} }{\mu(x^1_{i,1})}\\
		& = \alpha_i.
	\end{array}
\end{align*}
and 
\begin{align*}
	\begin{array}{ll}
		h_{\phi^n,\widehat{w}_n}(x^1_{i,j+1}) & = \displaystyle \frac{|\widehat{w}_n(x^1_{i, j+n+1})|^{2}\mu(x^1_{i, n+1})}{\mu(x^1_{i,j+1})}\\
		& = 1.
			\end{array}
\end{align*}
Therefore, \eqref{eqn4} is valid. 
\end{proof}
\begin{remark}\label{R1}
	The moment sequence of the Cauchy dual $C^{'}_\phi$ of the composition operator satisfies \eqref{equ1} is obtained by averaging the pointmass tranport quantities against $|f|^2$. So, the pointwise sequence 
	$ \left\{\lvert h_{\phi^n,\widehat{w}_n}(x) \rvert \right\}_{n=0}^{\infty}$ control local behaviour, while norm sequnce 
	$\left\{\lVert {C^{'}}^n_\phi f\rVert^2 \right\}_{n=0}^{\infty}$ are their global $L^2$-avarage. Therefore, to check the Stiletjes moment property of $\left\{\lVert {C^{'}}^n_\phi f\rVert^2 \right\}_{n=0}^{\infty}$, we use the corresponding property of 
		$\left\{ h_{\phi^n,\widehat{w}_n}(x)  \right\}_{n=0}^{\infty}$ \cite{BJJS, ZJJK }.
\end{remark}
\begin{theorem}\label{Tm2}
	Assume that \eqref{equ1} holds with $\kappa=1$ and $C_\phi \in B(L^2(\mu))$ is quasi-$2$-isometry. Then the Cauchy dual $C^{'}_\phi$ of $C_\phi$ is subnormal.
\end{theorem}
\begin{proof}
Since $C_\phi \in B(L^2(\mu))$ is quasi-$2$-isometry staisying \eqref{equ1} with $\kappa=1$, it follows from Lemma \ref{lm2} that its Cauchy dual is given by $$C^{'}_\phi = w_\phi C_\phi,$$ 
where $w_\phi= \displaystyle \frac{1}{h \circ \phi}$. Hence, $C^{'}_\phi$ is a weighted composition operator. Applying Lemma \ref{lm3} together with Remark \ref{R1}, we conclude that $C^{'}_\phi$ is subnormal if and only if the sequence $\left\{ h_{\phi^n,\widehat{w}_n}(x)  \right\}_{n=0}^{\infty}$ is a Stieltjes moment sequence.

Now consider 
\begin{align*}
	\begin{array}{ll}
		h_{\phi^n,\widehat{w}_n}(x_1) & = \displaystyle \alpha^{2n-1}+\frac{  \sum_{i=1}^{\eta_1} c_i \alpha_{i}^{2} \sum_{j=2}^{n} \alpha^{2(n+1-j)}	}{\mu(x_1)}\\
		& = \displaystyle \alpha^{2n-1}+\frac{\sum_{i=1}^{\eta_1} c_i \alpha_{i}^{2} \alpha^2}{\mu(x_1)} \left(\frac{(\alpha^2)^{n-1}-1}{\alpha^2-1}\right)\\
		& = \displaystyle \left[ \frac{1}{\alpha} + \frac{\beta}{(\alpha^2-1)\mu(x_1)} \right]\alpha^{2n} + \frac{\beta \alpha ^2}{(1-\alpha^2)\mu(x_1)},
	\end{array}
\end{align*}
where $\beta = \sum_{i=1}^{\eta_1} c_i \alpha_{i}^{2}.$ So, if we choose a measure 
$$  \nu =\displaystyle \left[ \frac{1}{\alpha} + \frac{\beta}{(\alpha^2-1)\mu(x_1)} \right]\delta_{\alpha^{2n}} + \frac{\beta \alpha ^2}{(1-\alpha^2)\mu(x_1)}\delta_1,$$
then $\displaystyle \int_{0}^{\infty} t^n d\nu(t) = \displaystyle \left[ \frac{1}{\alpha} + \frac{\beta}{(\alpha^2-1)\mu(x_1)} \right]\alpha^{2n} + \frac{\beta \alpha ^2}{(1-\alpha^2)\mu(x_1)} = h_{\phi^n,\widehat{w}_n}(x_1). $

Since 
$$h_{\phi^n,\widehat{w}_n}(x^1_{i,1}) =\alpha_{i},~~ i \in J_{[1, \eta_1]}$$
and 
$$h_{\phi^n,\widehat{w}_n}(x^1_{i,j+1})=1,~~ i \in J_{[1, \eta_1]}, j \in \mathbb{N},$$ it follows that the sequence $\left\{ h_{\phi^n,\widehat{w}_n}(x)  \right\}_{n=0}^{\infty}$
 is a Stieltjes moment sequence. 
\end{proof}
The following results provide a characterization of the $\Delta_{C_\phi}$-regularity of quasi-$2$-isometric composition operators on directed graphs satisfying \eqref{equ1}, corresponding to the cases $\kappa = 1$ and $\kappa = 2$.

\begin{theorem}\label{Tm3}
Suppose that \eqref{equ1} holds with $\kappa=1$ and $C_\phi \in B(L^2(\mu))$ is quasi-$2$-isometry. Then $C_\phi$ is  $\Delta_{C_\phi}$-regular if and only if  $h(x_1)= h(x^1_{i,1}), $ and $ \mu(x^1_{i,j+1})\geq  \mu(x^1_{i,1}) ~~~\textrm{for all}~~ i \in J_{[1, \eta_1]}, j\in \mathbb{N}.$
\end{theorem}
\begin{proof}
	Given that \eqref{equ1} holds with $\kappa=1$. Since $C_\phi \in B(L^2(\mu))$ is quasi-$2$-isometry, it follows that $\mu(x^1_{[i,j+1]})=c_i,$ a constant for  $i\in J_{{1,\eta_1}}, j\in \mathbb{N}.$ Moreover,  $\mu(x_1)$ and $\mu(x^1_{i,1}), i\in J_{[1,\eta_1]}$, may be chosen arbitrarily as positive numbers.
	
	To check the $\Delta_{C_\phi}$--regularity of $C_\phi$, we first find $\Delta_{C_\phi}C_\phi $ and $\Delta_{C_\phi}^{1/2}C_\phi \Delta_{C_\phi}^{1/2} $. Let $f \in L^2(\mu)$. Then  for $x \in X$, 
	\begin{align*}
		\begin{array}{ll}
			\Delta_{C_\phi}C_\phi f(x)  & = C^{*}_\phi C_\phi (f\circ \phi)(x) \\	
									& = (h-1)(x)(f\circ \phi)(x),
		\end{array}
	\end{align*}
where 
\begin{align*}
	(h-1)(x) = \left\{
	\begin{array}{ll}
		\displaystyle \frac{\sum_{i=1}^{\eta_1}\mu(x^1_{i,1})}{\mu(x_1)} & \text{if} \quad x=x_1\\
		\displaystyle \frac{c_i - \mu(x^1_{i,1})}{\mu(x^1_{i,1})} & \text{if} \quad x= x^1_{i,1}, ~~ i \in J_{[1, \eta_1]}\\
		0 &  \text{if} \quad x= x^1_{i,j+1}, ~~ i \in J_{[1, \eta_1]},~~ j \in \mathbb{N}.
	\end{array} \right.
\end{align*}
This implies 
\begin{align}\label{eqn6}
	\Delta_{C_\phi}C_\phi f = (h(x_1)-1)f(x_1) \chi_1 + \displaystyle \sum_{i=1}^{\eta_1}(h(x^1_{i,1})-1)f(x_1) \chi^1_{i,1}, 
\end{align}

and 
$$ \Delta_{C_\phi} f = (h(x_1)-1)f(x_1) \chi_1 + \displaystyle \sum_{i=1}^{\eta_1}(h(x^1_{i,1})-1)f(x^1_{i,1}) \chi^1_{i,1}, $$

where $ \chi_1$ and $\chi^1_{i,1}$ are the characteristic function of $\{x_1\}$ and $\{x^1_{i,1}\}$ for $i \in J_{[1, \eta_1]}$ respectively.

Observe that the operator $\Delta_{C_\phi}^{1/2}$ exists precisely when $\Delta_{C_\phi} $ is positive. Since this condition is equivalent to $c_i \geq \mu(x^1_{i,1}), ~~  i \in J_{[1, \eta_1]}. $ Therefore, $\Delta_{C_\phi}^{1/2}$ exists if and only if $ \mu(x^1_{i,j+1})\geq  \mu(x^1_{i,1}) ~~\textrm{for all}~~ i \in J_{[1, \eta_1]}, j\in \mathbb{N}.$

Now  consider 
$\Delta_{C_\phi}^{1/2}C_\phi \Delta_{C_\phi}^{1/2} f $
\begin{align}\label{eqn7}
	\begin{array}{ll}
	& = \Delta_{C_\phi}^{1/2}C_\phi \displaystyle \left[  \sqrt{(h(x_1)-1)}f(x_1) \chi_1 + \displaystyle \sum_{i=1}^{\eta_1}\sqrt{(h(x^1_{i,1})-1)}f(x^1_{i,1}) \chi^1_{i,1}   \right]  \\
	 & = \Delta_{C_\phi}^{1/2} \displaystyle \left[  \sqrt{(h(x_1)-1)}f(x_1) \left( \chi_1 + \sum_{i=1}^{\eta_1} \chi^1_{i,1} \right) + \displaystyle \sum_{i=1}^{\eta_1}\sqrt{(h(x^1_{i,1})-1)}f(x^1_{i,1}) \chi^1_{i,2}    \right] \\
	 & = (h(x_1)-1)f(x_1)\chi_1 + \displaystyle \sqrt{(h(x_1)-1)}f(x_1) \sum_{i=1}^{\eta_1} \sqrt{(h(x^1_{i,1})-1)} \chi^1_{i,1}  \\
	\end{array}
\end{align}
Hence, from \eqref{eqn6} and  \eqref{eqn7},the characterization of the $\Delta_{C_\phi}$-regularity of $C_\phi$ follows.
\end{proof}
\begin{example}\label{eg2}
Let $\kappa=1$ and suppose that $C_\phi \in B(L^2(\mu))$ is a quasi-$2$-isometry. Then $C_\phi$ need not be $\Delta_{C_\phi}$-regular in general. 
	Indeed, choose $\eta_1 = 2, c_1 = c_2 = 1, \mu(x_1) = 2 ~~\text{and} ~~\mu(x^1_{1,1})= \mu(x^1_{2,1}) = 1/2 .$ Then $h(x_1)= 3/2$ and $h(x^1_{1,1})= h(x^1_{2,1})=2 .$ Since $h(x_1)\neq h(x^1_{i,1})$ for $i=1,2$, it follows from Theorem~\ref{Tm3} that $C_\phi$ is not $\Delta_{C_\phi}$-regular.
	On the other hand, let $\eta_1 = 2, c_1 = c_2 = 1, \mu(x_1) = 1 ~~~\text{and}~~ \mu(x^1_{1,1})= \mu(x^1_{2,1}) = 1/2 .$ Then $h(x_1)= 2 $ and $h(x^1_{1,1})= h(x^1_{2,1})=2 .$ Hence, by Theorem~\ref{Tm3}, $C_\phi$ is $\Delta_{C_\phi}$-regular.
	Therefore, even if $C_\phi$ is a quasi-$2$-isometry,  we observed that $C_\phi$ may or may not be $\Delta_{C_\phi}$-regular.
	
\end{example}

\begin{theorem}\label{Tm4}
	Suppose that \eqref{equ1} holds with $\kappa=2$ and $C_\phi \in B(L^2(\mu))$ is quasi-$2$-isometry. Then $C_\phi$ is  $\Delta_{C_\phi}$-regular if and only if $h(x_1)= h(x_2) = h(x^r_{i,1}) $ and $ \mu(x^r_{i,j+1})\geq  \mu(x^r_{i,1}) ~~\textrm{for all} ~~r \in J_{[1, \kappa]},  i \in J_{[1, \eta_r]}, j\in \mathbb{N}.$ 
\end{theorem}
\begin{proof}
	Let $\kappa=2$. Assume that \eqref{equ1} holds and $C_\phi \in B(L^2(\mu))$ is quasi-$2$-isometry. Then $\mu(x^r_{i,j+1})=c^{(r)}_i $ for $r\in J_{[1, \kappa]}, ~~ i\in J_{[1, \eta_r]}, ~~ j\in \mathbb{N},$ $\displaystyle \sum_{i=1}^{\eta_{r}} c^{(r)}_i = c^{(r)}< \infty $ and 
	$$
-2\mu(x_1)+2\mu(x_2) + 2 \sum_{i=1}^{\eta_{1}}\mu(x^1_{i,1}) - 2 \sum_{i=1}^{\eta_{2}}\mu(x^2_{i,1})- c^{(1)} + c^{(2)} = 0.
	$$
	To verify the $\Delta_{C_\phi}$-regularity of $C_\phi,$  let us consider $h-1$, where $h$ is the Radon - Nikodym derivative of $\mu \circ \phi$ with respect to $\mu$. 
	\begin{align}\label{eqn8}
		(h-1)(x) = \left\{
		\begin{array}{ll}
			\displaystyle \frac{\mu(x_2) - \mu(x_1) + \sum_{i=1}^{\eta_1}\mu(x^1_{i,1})}{\mu(x_1)} & \text{if} \quad x=x_1\\
			\displaystyle \frac{\mu(x_1) - \mu(x_2) + \sum_{i=1}^{\eta_2}\mu(x^2_{i,1})}{\mu(x_2)} & \text{if} \quad x=x_2\\
			\displaystyle \frac{c^{(r)}_i - \mu(x^r_{i,1})}{\mu(x^r_{i,1})} & \text{if} \quad x= x^r_{i,1}, ~~ r\in J_{[1, \kappa]}, ~~ i \in J_{[1, \eta_r]}\\
			0 &  \text{if} \quad x= x^1_{i,j+1}, ~~ i \in J_{[1, \eta_1]},~~ j \in \mathbb{N}.
		\end{array} \right.
	\end{align} 
By \eqref{eqn8}, we see that 
\begin{align}\label{eqn9}
	\Delta_{C_\phi} C_\phi f = (h(x_1)-1)f(x_2) \chi_1 + (h(x_2)-1)f(x_1) \chi_2 +  \sum_{r=1}^{\kappa}\sum_{i=1}^{\eta_r}(h(x^r_{i,1})-1)f(x_r) \chi^r_{i,1},
\end{align}
where $\chi_1, \chi_2,~~\text{and}~~\chi^r_{i,1}$ are the characteristic functions of $\{x_1\}, \{x_2\}, ~~\text{and}~~ \{x^r_{i,1}\}$ respectively  for $ r\in J_{[1, \kappa]}, ~~ i \in J_{[1, \eta_r]}$.
Also, 
$\Delta_{C_\phi}^{1/2}C_\phi \Delta_{C_\phi}^{1/2} f $
\begin{align}\label{eqn10}
	\begin{array}{ll}
		& = \Delta_{C_\phi}^{1/2}C_\phi \displaystyle \left[  \sqrt{(h(x_1)-1)}f(x_1) \chi_1 + \sqrt{(h(x_2)-1)}f(x_2) \chi_2 \right] \\
		& \quad \quad \quad \quad \quad \quad \quad \quad \quad 
		+ \Delta_{C_\phi}^{1/2}C_\phi \displaystyle \left[ \sum_{r=1}^{\kappa} \sum_{i=1}^{\eta_r}\sqrt{(h(x^r_{i,1})-1)}f(x^r_{i,1}) \chi^r_{i,1}   \right]  \\
		& = \Delta_{C_\phi}^{1/2} \displaystyle \left[  \sqrt{(h(x_1)-1)}f(x_1) \left( \chi_2 + \sum_{i=1}^{\eta_1} \chi^1_{i,1} \right) + \sqrt{(h(x_2)-1)}f(x_2) \left( \chi_1 + \sum_{i=1}^{\eta_2} \chi^2_{i,1} \right) \right] \\            
		& \quad \quad \quad \quad \quad \quad  + \Delta_{C_\phi}^{1/2}  \left[ \displaystyle \sum_{r=1}^{\kappa}\sum_{i=1}^{\eta_r}\sqrt{(h(x^r_{i,1})-1)}f(x^r_{i,1}) \chi^1_{i,2}    \right] \\
		& = \sqrt{(h(x_1)-1)}f(x_1) \left[ \sqrt{(h(x_2)-1)}\chi_2 + \sum_{i=1}^{\eta_1} \sqrt{(h(x^1_{i,1})-1)} \chi^1_{i,1} \right]  \\
		& + \sqrt{(h(x_2)-1)}f(x_2) \left[\sqrt{(h(x_1)-1)}\chi_1 + \sum_{i=1}^{\eta_2} \sqrt{(h(x^2_{i,1})-1)} \chi^2_{i,1} \right] . 
	\end{array}
\end{align}
	The required result follows from  \eqref{eqn9} and  \eqref{eqn10} 
\end{proof}

In what follows, we investigate and characterize the kernel condition for quasi-$2$-isometric composition operators $C_\phi \in B(L^2(\mu))$ satisfying \eqref{equ1}. 
\begin{theorem}\label{Tm5}
	 Assume that \eqref{equ1} holds and that $C_\phi \in B(L^2(\mu))$ is a quasi-$2$-isometry. Then the following statements are true; 
	 \begin{enumerate}
	 	\item[(a)]  For $\kappa = 1$, $C_\phi$ satisfies the kernel condition if and only if   $ h(x_1) = h(x^1_{i,1}), ~~ i\in J_{[1, \eta_{1}]}.$
	 	\item[(b)]   For $\kappa =2$, $C_\phi$ satisfies the kernel condition if and only if   $ h(x_1) = h(x^2_{i,1}), ~~ i\in J_{[1, \eta_{2}]}$ and 
	 	$ h(x_2) = h(x^1_{i,1}), ~~ i\in J_{[1, \eta_{1}]}.$
	 \end{enumerate}	  
\end{theorem}
\begin{proof}
(a) Since  $C_\phi \in B(L^2(\mu))$ is a quasi-$2$-isometry, it follows that $\mu(x^1_{i, j+1})= c^{(1)}_i,$ where $c^{(1)}_i$ is consatant  for each $  i\in J_{[1, \eta_{1}]}, ~~~ j \in \mathbb{N},$ with $~~~ \displaystyle \sum_{i=1}^{\eta_{1}} c^{(1)}_i = c^{(1)} < \infty $. Moreover,  $\mu(x_1)$ and  $\mu(x^1_{i, 1}) $ may be chosen as  arbitrarily as positive numbers. Therefore, by  \cite[Proposition 4.8]{ZJJK } we get   $C_\phi$ satisfies the kernel condition if and only if   $ h(x_1) = h(x^1_{i,1}), ~~ i\in J_{[1, \eta_{1}]}.$

(b) For $\kappa = 2,$ the necessary and sufficient condition for $C_\phi$ to satisfy the kernel condition can be readily obtained by examining the atoms of  $\phi^{-1}\mathcal{A}$ together with the fact that $C_\phi$ is a  quasi-$2$-isometry. 
\end{proof}

\begin{remark}
	For $\kappa = 1$, the kernel condition and $\Delta_{C_\phi}$-regularity of a quasi-$2$-isometric composition operator  $C_\phi$ satisfying \eqref{equ1} are equivalent. In contrast, when $\kappa = 2$, $\Delta_{C_\phi}$-regularity  implies the kernel condition; however, the following example shows that the  converse  does not is hold in general. 
\end{remark}
	\begin{example}\label{eg3}
		Let $\kappa = 2$, $\eta_{1} = \eta_{2} = 1$. Take $\mu(x_1)=\mu(x_2) = 1, c^{(1) }= c^{(1)}_1 = 2/3, c^{(2) }= c^{(2)}_1 = 1/2, \mu(x^1_{1,1}) = 1/2$ and $ \mu(x^2_{1,1}) = 1/3.$ 
		
		Then 
		$$                                         
		h(x_1) = \displaystyle \frac{\mu(x_2) + \mu(x^1_{1,1})}{\mu(x_1)} = \frac{3}{2}, \quad h(x^2_{1,1}) = \frac{c^{(2)}_1}{\mu(x^2_{1,1})}= \frac{3}{2},
		$$
		
		$$
			h(x_2) = \displaystyle \frac{\mu(x_1) + \mu(x^2_{1,1})}{\mu(x_2)} = \frac{4}{3}, \quad
			h(x^1_{1,1}) = \frac{c^{(1)}_1}{\mu(x^1_{1,1})}= \frac{4}{3}.
		$$
		Therefore, by Theorem \ref{Tm5},  $C_\phi$ satisfies kernel condition, whereas  $\Delta_{C_\phi}$-regularity does not hold.
	\end{example}


\end{document}